\documentclass[11pt,lenq]{amsart}
\usepackage{amsmath}
\usepackage{amscd}
\usepackage{amssymb}
\usepackage{amsbsy}
\usepackage{amsfonts}
\usepackage{color}
\usepackage{latexsym}
\usepackage{graphics}
\usepackage{graphicx}
\usepackage{amsmath,amscd,latexsym}
\usepackage{array}
\usepackage{mathtools}
\usepackage{enumerate}
\theoremstyle{plain}
\newtheorem{theorem}{Theorem}[section]
\newtheorem{proposition}[theorem]{Proposition}

\newtheorem{corollary}[theorem]{Corollary}
\newtheorem{remark}[theorem]{Remark}
\newtheorem{definition}[theorem]{Definition}

\newtheorem{main theorem}[theorem]{Main Theorem}

\newlength\savewidth

\newcommand{\svert}{\,|\,}

\newcommand{\llangle}{\langle\!\langle}
\newcommand{\rrangle}{\rangle\!\rangle}

\newcommand{\im}{{\rm im}}


\begin{document}
\title[A non-Hopfian relatively hyperbolic group with respect to a Hopfian subgroup]
{A non-Hopfian relatively hyperbolic group with respect to a Hopfian subgroup}

\author{Jan Kim}
\address{Department of Mathematics\\
Pusan National University \\
San-30 Jangjeon-Dong, Geumjung-Gu, Pusan, 609-735, Korea}
\email{jankim@pusan.ac.kr}

\author{Donghi Lee}
\address{Department of Mathematics\\
Pusan National University \\
San-30 Jangjeon-Dong, Geumjung-Gu, Pusan, 609-735, Korea}
\email{donghi@pusan.ac.kr}



\begin{abstract}
We produce an example demonstrating
that every finitely generated relatively hyperbolic group
with respect to a collection of Hopfian subgroups need not be Hopfian.
This answers a question of Osin \cite[Problem 5.5]{Osin} in the negative.
\end{abstract}

\maketitle

\section{Introduction}
Recall that a group $G$ is {\em Hopfian}
if every epimorphism $G \rightarrow G$ is an automorphism.
Recall also that a group $G$ is {\em residually finite}
if for every $g \in G \setminus \{1\}$,
there is some finite group $P$ and an epimorphism $\psi : G \rightarrow P$ so that $\psi(g)\neq 1$.
Inspirited by well-known questions about ordinary hyperbolic groups,
Osin~\cite[Problems 5.5 and 5.6]{Osin} asked the following questions.

\begin{itemize}
\item If a finitely generated group $G$ is hyperbolic relative to a collection of Hopfian subgroups
$\{H_1, \dots, H_m\}$,
does it follow that $G$ is Hopfian?

\item If a group $G$ is hyperbolic relative to a collection of residually finite subgroups
$\{H_\lambda\}_{\lambda \in \Lambda}$,
does it follow that $G$ is residually finite?
\end{itemize}

Later, Osin~\cite{Osin4} proved that the second question is equivalent to
Gromov's famous open question of whether every hyperbolic group
is residually finite.
The Hopf property and the residual finiteness property have a close connection.
In particular, Mal'cev~\cite{Malcev} proved that
every finitely generated residually finite group is Hopfian.
Mal'cev's result provides a useful tool to prove that a certain finitely generated group is non-residually finite.
The Hopf properties of torsion-free hyperbolic groups,
toral relatively hyperbolic groups,
hyperbolic groups with torsion,
lacunary hyperbolic groups and finitely presented $C'(1/6)$ or $C'(1/4)$-$T(4)$ small cancellation groups
were verified by many authors (see \cite{Coulon-Guirardel, Groves, Rein-Weid,Sela,Strebel}).
In contrast, Wise~\cite{Wise2} constructed a non-Hopfian $CAT(0)$-group.

On the other hand, there is another property related to the Hopf property.
A group $G$ is called {\em equationally noetherian} if for every system of equations in $G$,
there exists a finite subsystem that has the same set of solutions.
It is well-known that every finitely generated equationally noetherian group is Hopfian.
Reinfeldt and Weidmann~\cite{Rein-Weid} proved that every hyperbolic group is equationally noetherian.
Also, for relatively hyperbolic groups, Groves and Hull~\cite{Groves_Hull} proved that if a group $G$ is hyperbolic relative to a collection of equationally noetherian subgroups, then $G$ is itself equationally noetherian.
However, it has been unknown up to the present whether every finitely generated group
that is hyperbolic relative to a collection of Hopfian subgroups is Hopfian.

The main result of this paper is the following.
This solves Osin's first question mentioned above in the negative.

\begin{theorem}
\label{thm:main_theorem}
Let $\mathbf{H}_0$ be the group given by the presentation
\begin{equation}
\label{equ:H_0_presentation}
\mathbf{H}_0=\langle b, c \svert b^2=c^9=1, \ b^{-1}cb=c^{-1} \rangle,	
\end{equation}
and take successively two HNN-extensions from $\mathbf{H}_0$ as follows:
\begin{subequations}
\begin{align}
\label{equ:H_1_presentation}
\mathbf{H}_1&=\langle \mathbf{H}_0, s \svert s^{-1}bs=bc^{-3}, \ s^{-1}cs=c \rangle;\\
\label{equ:H_2_presentation}
\mathbf{H}_2&=\langle \mathbf{H}_1, t \svert t^{-1}st=s^3 \rangle.
\end{align}
\end{subequations}
Next, form the free product $\mathbf{H}=\mathbf{H}_2 \ast \langle e, f \svert \emptyset \rangle$.
Finally, letting $\langle a \rangle$ be an infinite cyclic group,
take successively two multiple HNN-extensions from $\mathbf{H} \ast \langle a \rangle$ as follows:
\begin{subequations}
\begin{align}
\label{equ:K_presentation}
\mathbf{K}&=\langle \mathbf{H} \ast \langle a \rangle, u, v \svert u^{-1}(bacb^{-1})u=a, \  v^{-1}av=tst^{-1} \rangle;\\
\label{equ:G_presentation}
\mathbf{G}&=\langle \mathbf{K}, x, y \svert x^{-1}ux=c^3ec^3e^{-1}, \ y^{-1}vy=c^3fc^3f^{-1} \rangle.	
\end{align}
\end{subequations}
Then $\mathbf{G}$ is a non-Hopfian group which is hyperbolic relative to the Hopfian subgroup $\mathbf{H}$.
In more detail, the following hold.
\begin{enumerate}[\rm (i)]
\item
$\mathbf{K}$ is hyperbolic relative to the subgroup $\mathbf{H}$.	
	
\item
$\mathbf{G}$ is hyperbolic relative to the subgroup $\mathbf{H}$.
	
\item
$\mathbf{G}$ is a non-Hopfian group.
	
\item
$\mathbf{H}_2$ is a Hopfian group, and thus $\mathbf{H}$ is a Hopfian group.
\end{enumerate}
\end{theorem}

\begin{remark}
{\rm
\begin{enumerate}[\rm (1)]
\item
The group $\mathbf{G}$ can be regarded as a relatively hyperbolic group with respect to the subgroup $\mathbf{H}_2$.
The reason goes as follows.
Since $\mathbf{H}$ is the free product of $\mathbf{H}_2$ and $\langle e, f \rangle$,
clearly $\mathbf{H}$ is hyperbolic relative to the collection of subgroups $\{\mathbf{H}_2, \langle e, f \rangle \}$.
Here, since every finitely generated free group is hyperbolic,
$\mathbf{H}$ is hyperbolic relative to the subgroup $\mathbf{H}_2$.
This together with (ii) yields that $\mathbf{G}$ is hyperbolic relative to the subgroup $\mathbf{H}_2$.

\item The subgroup $\mathbf{H}$ is non-residually finite.
Indeed, for any finite group $P$ and for any epimorphism $\psi$
from $\mathbf{H}$ to $P$, $\psi(c^3)=1$.
The reason can be seen as follows.
From the defining relation $t^{-1}st=s^3$ of $\mathbf{H}$,
it follows that $\psi(s)$ and $\psi(s)^3$ have the same order,
so that the order of $\psi(s)$ is relatively prime to $3$, say $m$.
Also from the defining relation $s^{-1} b s=b c^{-3}$ of $\mathbf{H}$,
it follows that $b^{-1} s b= sc^3$ in $\mathbf{H}$, so that $\psi(sc^3)^m=1$.
Here, since $\psi(s)$ and $\psi(c)$ commute with each other, $\psi(c^3)^m=1$.
On the other hand, since $c^9=1$ in $\mathbf{H}$,
$\psi(c^3)^3=1$, which together with $\psi(c^3)^m=1$
finally yields $\psi(c^3)=1$.
\end{enumerate}
}
\end{remark}

This paper is organized as follows.
In Section~\ref{sec:preliminaries}, we recall necessary definitions and known results to be used throughout this paper.
The proof of Theorem~\ref{thm:main_theorem} is contained in Sections~\ref{sec:K_relatively_hyperbolic}--\ref{sec:Hopficity}.
In Section~\ref{sec:K_relatively_hyperbolic}, by using Osin's theorem
concerning the unique maximal elementary subgroups of hyperbolic elements
in relatively hyperbolic groups,
we first prove that the free product
$\mathbf{H} \ast \langle a \rangle$ is hyperbolic relative to the collection of subgroups
$\{\mathbf{H}, \langle a \rangle, \langle ac \rangle \}$.
And then by successively using Osin's combination theorem for relatively hyperbolic groups,
we show that $\mathbf{K}$ is hyperbolic relative to $\mathbf{H}$.
In Section~\ref{sec:G_relatively_hyperbolic},
again by using Osin's theorem about unique maximal elementary subgroups,
we show that the peripheral structure of $\mathbf{K}$ can be extended to the collection of subgroups
$\{\mathbf{H}, \langle u \rangle, \langle v \rangle \}$.
At this point, by using Osin's combination theorem twice,
we obtain that $\mathbf{G}$ is hyperbolic relative to $\mathbf{H}$.
In Section~\ref{sec:non_Hopficity},
we show that $\mathbf{G}$ is non-Hopfian by constructing a particular surjective, but not injective, endomorphism of $\mathbf{G}$.
To be more precise, the endomorphism of $\mathbf{G}$ induced by the mapping $b \mapsto b$, $c \mapsto c^3$, $s \mapsto s^3$,
$t \mapsto t$, $e \mapsto e$, $f \mapsto f$, $a \mapsto s$, $u \mapsto 1$, $v \mapsto 1$, $x \mapsto x$ and $y \mapsto y$
is shown to be surjective but not injective.
Finally, Section~\ref{sec:Hopficity} is devoted to the proof of that $\mathbf{H}$ is Hopfian,
in which Bass-Serre theory plays a crucial role.

\section*{Acknowledgements}
The authors would like to express their sincere gratitude to the anonymous referees
for their kind comments and suggestions,
which helped them in improving the exposition of the manuscript.
In particular, the suggestion to use Bass-Serre theory
in the proof of the fact that $\mathbf{H}_2$ is Hopfian
enabled many arguments to be significantly simplified,
and the inclusion of an alternative exposition was greatly appreciated,
as it highlighted aspects that the authors had not previously realized.
The second author was supported by Basic Science Research Program
through the National Research Foundation of Korea(NRF) funded
by the Ministry of Education, Science and Technology(2020R1F1A1A01071067).

\section{Preliminaries}
\label{sec:preliminaries}

In this section, we recall necessary definitions, notation and known results to be used throughout this paper.

\subsection{Relatively hyperbolic groups}
\label{subsec:relatively_hyperbolic_groups}

In this paper, we adopt Osin's definition~\cite{Osin2}
among many equivalent definitions of relatively hyperbolic groups.

Let $G$ be a group, $\mathbb{H}=\{H_{\lambda}\}_{\lambda \in \Lambda}$
a collection of subgroups of $G$, and $X$ a subset of $G$.
Suppose that $X$ is a relative generating set for $(G, \mathbb{H})$,
namely, $G$ is generated by the set $\big(\bigcup_{\lambda \in \Lambda} H_{\lambda} \big) \cup X$
(for convenience, we assume that $X=X^{-1}$).
Then $G$ can be regarded as the quotient group of the free product
\[
F=(\ast_{\lambda \in \Lambda} \tilde{H}_\lambda) \ast F(X),
\]
where the groups $\tilde{H}_\lambda$ are isomorphic copies of $H_\lambda$,
and $F(X)$ is the free group generated by $X$.
Let $\mathcal{H}$ be the disjoint union
\[
	\mathcal{H}=\bigsqcup_{\lambda \in \Lambda} (\tilde{H}_\lambda \setminus \{1\}).
\]
For every $\lambda \in \Lambda$, we denote by $S_{\lambda}$ the set of all words over the alphabet
$\tilde{H}_\lambda \setminus \{1\}$ that represent the identity in $F$.
Let $\mathcal{S}$ be the disjoint union
\[
\mathcal{S}=\bigsqcup_{\lambda \in \Lambda} S_{\lambda}.
\]
Then we may describe $G$ as a {\it relative presentation}
\begin{equation}
	\label{equ:relative_presentation}
	\langle X, \mathcal{H} \svert \mathcal{S}, \mathcal{R} \rangle
\end{equation}
with respect to the collection of subgroups $\{H_{\lambda}\}_{\lambda \in \Lambda}$,
where $\mathcal{R} \subseteq F$.
If both the sets $\mathcal{R}$ and $X$ are finite,
relative presentation (\ref{equ:relative_presentation}) is said to be {\it finite}
and the group $G$ is said to be {\it finitely presented relative to the
	collection of subgroups $\mathbb{H}$}.

For every word $w$ over the alphabet $X \cup \mathcal{H}$
representing the identity in the group $G$,
there exists an expression
\begin{equation}
	\label{equ:identity}
	w=_F \prod_{i=1}^k f_i^{-1} R_i f_i
\end{equation}
with the equality in the group $F$, where $R_i \in \mathcal{R}$ and $f_i \in F$ for $i=1, \dots, k$.
The smallest possible number $k$ in a presentation of the form (\ref{equ:identity})
is called the {\it relative area} of $w$ and is denoted by $Area^{rel}(w)$.

\begin{definition}[Relatively hyperbolic groups]
	\label{def:relatively hyperbolic_group}	
	{\rm
		A group $G$ is said to be {\it hyperbolic relative to a collection of subgroups $\mathbb{H}$}
		if $G$ admits a relatively finite presentation (\ref{equ:relative_presentation})
		with respect to $\mathbb{H}$
		satisfying a {\it linear relative isoperimetric inequality}.
		That is, there is a constant $C>0$ such that for any cyclically reduced word $w$
		over the alphabet $X \cup \mathcal{H}$
		representing the identity in $G$,
		we have
		\[
		Area^{rel}(w) \le C \| w \|,
		\]
		where $\|w\|$ is the length of the word $w$.
		This definition is independent of the choice of the finite
		relative generating set $X$ and the finite set $\mathcal{R}$ in (\ref{equ:relative_presentation}).		
	}	
\end{definition}

\subsection{Unique maximal elementary subgroups}

Suppose that $G$ is hyperbolic relative to a collection of subgroups
$\mathbb{H}=\{H_{\lambda}\}_{\lambda \in \Lambda}$. Then we refer to the collection $\mathbb{H}$ as a {\it peripheral structure} of $G$,
and any element in $\mathbb{H}$ as a {\it peripheral subgroup} of $G$.

An element is called {\em hyperbolic} if it has infinite order and it is not conjugate
to any element of a peripheral subgroup of $G$.
Due to Osin~\cite{Osin5}, there is a well-known example of subgroups which may be added to enlarge peripheral structures.

\begin{theorem}[{\cite[Theorem~4.3, Corollary~1.7]{Osin5}}]
\label{thm:hyperbolically_embedded}
Let $G$ be hyperbolic relative to a collection of subgroups $\mathbb{H}$.
Then for any hyperbolic element $g \in G$, $G$ is hyperbolic relative to $\mathbb{H} \cup \{E(g)\}$, where $E(g)$ is the unique maximal elementary subgroup containing $g$
defined as follows:
\[
E(g)=\{f \in G : fg^nf^{-1} = g^{\pm n} \ \textrm{for some} \ n \in \mathbb{N}\}.
\]
\end{theorem}

\subsection{Osin's combination theorem}
\label{subsec:combination_theorems}

We recall one of Osin's combination theorems for relatively hyperbolic groups.
Earlier, Dahmani~\cite{Dahmani} proved the following combination theorem for finitely generated groups.
In fact, applying Dahmani's combination theorem is sufficient for our purposes in this paper,
but we introduce Osin's combination theorem
in order to match with the definition of relatively hyperbolic groups stated above.

\begin{theorem}[{\cite[Corollary 1.4]{Osin2}}]
\label{thm:osin_combination}
Suppose that a group $G$ is hyperbolic relative to a collection
of subgroups $\mathbb{H}=\{H_{\lambda}\}_{\lambda \in \Lambda}$.
Assume in addition that there exists a monomorphism $\iota: H_{\mu} \rightarrow H_{\nu}$
for some $\mu \neq \nu \in \Lambda$,
and that $H_{\mu}$ is finitely generated.
Then the HNN-extension
\[
G^*=\langle G, t \svert t^{-1} h t=\iota(h), \ h \in H_{\mu} \rangle
\]
is hyperbolic relative to the collection $\mathbb{H} \setminus \{H_{\mu}\}$.
\end{theorem}

\subsection{Bass-Serre trees for HNN-extensions}

We recall some basic concepts of Bass-Serre theory (see~\cite{Bass, Serre}).
A {\em graph of groups} $(\mathcal{A}, X)$ consists of
a connected graph $X$ and a collection of groups indexed by the vertices and edges of $X$,
and a family of monomorphisms from the edge groups to the adjacent vertex groups.
For each spanning tree $T$ in $X$,
one can canonically associate a unique group, called the {\em fundamental group} and denoted $\pi_1(\mathcal{A}, T)$.
Here, it turns out that the fundamental group $\pi_1(\mathcal{A}, T)$ is independent of the choice of a spanning tree $T$,
so that we simply write $\pi_1(\mathcal{A})$ instead of $\pi_1(\mathcal{A}, T)$.
The fundamental group $\pi_1(\mathcal{A})$
admits an orientation-preserving action on a tree $\Gamma$
such that the quotient graph $\mathcal{A} / \pi_1(\mathcal{A})$ is isomorphic to $X$.
Such a tree is called a {\em Bass-Serre tree} of $\mathcal{A}$.

On the other hand, given a graph of groups $(\mathcal{A}, X)$ with the fundamental
group $G \cong \pi_1(\mathcal{A})$, where $G$ is an HNN-extension,
one can construct a Bass-Serre tree for $G$ due to the following theorem (see, for example, \cite{Yang}).
This result plays an important role in the proof of
Proposition~\ref{prop:phi_s_1}.

\begin{theorem}[Bass-Serre trees for HNN-extensions]
\label{thm:Bass_Serre_HNN_extension}
Suppose that $G^*$ is an HNN-extension of a group $G$
with associated isomorphism $\iota$ between two subgroups $H$ and $K$,
that is,
\[
G^*=\langle G, t \svert t^{-1} h t = \iota(h), \ h \in H \rangle.
\]
Let $\mathcal{A}$ be a graph of groups
consisting of a single loop-edge $e$,
a single vertex $v=o(e)=t(e)$,
a vertex group $G$, an edge group $H$,
and the boundary monomorphisms $\alpha_e: H \rightarrow G$ and $\omega_e: H \rightarrow G$.
Then the fundamental group of $\mathcal{A}$ is clearly isomorphic to $G^*$.
On the other hand, let $\Gamma$ be a graph defined as follows.
\begin{enumerate}[\indent \rm (i)]
\item The vertex set $V$ consists of all cosets in $\{xG \, | \, x \in G^*\}$.
\item The edge set $E$ consists of all cosets in $\{xH \, | \, x \in G^*\}$.
\item The edge $xH \in E$ connects $xG$ and $xtG$.	
\end{enumerate}	
Then $\Gamma$ is a tree and
$G^*$ acts on $\Gamma$ without inversion by left multiplication
such that the quotient graph $\Gamma /G^*$ is isomorphic to $X$,
where $X$ is the underlying graph of $\mathcal{A}$.	
\end{theorem}

\section{Proof of Theorem~1.1({\rm i})}
\label{sec:K_relatively_hyperbolic}

Let $\mathbf{H}$ and $\mathbf{K}$ be the groups
defined in the statement of Theorem~\ref{thm:main_theorem}.
The aim of this section is to prove the relative hyperbolicity of $\mathbf{K}$ with respect to the subgroup
$\mathbf{H}$.

We start with the free product $\mathbf{H} \ast \langle a \rangle$,
which is clearly hyperbolic relative to the collection of subgroups
$\{\mathbf{H}, \langle a \rangle\}$.
Recall that
\begin{equation}
\label{equ:free_product}
\begin{aligned}
\mathbf{H} \ast \langle a \rangle=\langle b, c, s, t, a \svert
&b^2=c^9=1, \ b^{-1}cb=c^{-1}, \\
&s^{-1}bs=bc^{-3}, \ s^{-1}cs=c, \ t^{-1}st=s^3 \rangle.
\end{aligned}
\end{equation}
Clearly, $ac$ is a hyperbolic element in $\mathbf{H} \ast \langle a \rangle$
seen as a relatively hyperbolic group with peripheral structure $\{\mathbf{H}, \langle a \rangle\}$.
Moreover, we can prove the following

\medskip
\noindent {\bf Claim~A.}
{\it The unique maximal elementary subgroup $E(ac)$ of $\mathbf{H} \ast \langle a \rangle$
is precisely the infinite cyclic subgroup $\langle ac \rangle$.}

\begin{proof}
Suppose to the contrary that $E(ac) \setminus \langle ac \rangle \neq \emptyset$.
Among all such elements in $E(ac) \setminus \langle ac \rangle$,
we take an element in normal form, say $f$,
with shortest syllable length.
Here, by the {\em syllable length}, we mean
the total number of syllables
which are maximal subwords consisting entirely of letters from either $\mathbf{H}$ or $\langle a \rangle$.
For such $f$, clearly $f(ac)^{\pm n}f^{-1}=(ac)^n$ for some $n \in \mathbb{N}$.
Moreover, $f$ satisfies the following.
\begin{enumerate}[\ \rm (i)]
\item $f$ does not begin with $ah$ for any $1 \neq h \in \mathbf{H}$,
nor with $c^{-1}a'$ for any $1 \neq a' \in \langle a \rangle$;

\item $f$ does not end with $ha^{-1}$ for any $1 \neq h \in \mathbf{H}$,
nor with $a'c$ for any $1 \neq a' \in \langle a \rangle$.
\end{enumerate}
The reason goes as follows.
First, assume that $f$ begins with $ah$ for some $1 \neq h \in \mathbf{H}$,
that is, $f \equiv ah f_1$ in normal form.
Then it follows from the equality $f(ac)^{\pm n}f^{-1}=(ac)^n$ that
\[
(ac)(c^{-1}hf_1)(ac)^{\pm n} (f_1^{-1}h^{-1}c)(c^{-1}a^{-1})=(ac)^n,
\]
so that
\[
(c^{-1}hf_1)(ac)^{\pm n} (f_1^{-1}h^{-1}c)=(ac)^n;
\]
thus $c^{-1}hf_1 \in E(ac) \setminus \langle ac \rangle$.
Here, since $c^{-1}h \in \mathbf{H}$,
the element $c^{-1}hf_1$ has shorter syllable length
than $f$ does. This is a contradiction to our choice of $f$.
Next, assume that $f$ begins with $c^{-1}a'$ for some $1 \neq a' \in \langle a \rangle$,
that is, $f \equiv c^{-1}a' f_2$ in normal form.
Then
\[
(c^{-1}a^{-1})(aa'f_2)(ac)^{\pm n} (f_2^{-1}{a'}^{-1}a^{-1})(ac)=(ac)^n,
\]
and so
\[
(aa'f_2)(ac)^{\pm n} (f_2^{-1}{a'}^{-1}a^{-1})=(ac)^n.
\]
This means that $aa'f_2 \in E(ac) \setminus \langle ac \rangle$.
In addition, $aa'f_2$ has shorter syllable length
than $f$ does, since $aa' \in \langle a \rangle$.
This is also a contradiction to our choice of $f$.
So (i) holds.

For (ii), note that $f^{-1} \in E(ac) \setminus \langle ac \rangle$
with the same syllable length as that of $f$.
Also, if $f$ ends with $ha^{-1}$ for some $1 \neq h \in \mathbf{H}$,
or with $a'c$ for some $1 \neq a' \in \langle a \rangle$,
then
$f^{-1}$ begins with $ah^{-1}$ or with $c^{-1}{a'}^{-1}$.
But then by the same argument applied to $f$ in the proof of (i),
we reach a contradiction. So (ii) holds.

But then the expression $f(ac)^{\pm n}f^{-1}(ac)^{-n}$
cannot represent the identity element in $\mathbf{H} \ast \langle a \rangle$
by the normal form theorem for free products.
This contradiction completes the proof of the claim.
\end{proof}

The above Claim~A together with Theorem~\ref{thm:hyperbolically_embedded}
yields that $\mathbf{H} \ast \langle a \rangle$
is hyperbolic relative to the collection of subgroups
$\{\mathbf{H}, \langle a \rangle, \langle ac \rangle \}$.
Then due to Theorem~\ref{thm:osin_combination},
the group $\langle \mathbf{H} \ast \langle a \rangle, u \svert u^{-1}b(ac)b^{-1} u = a \rangle$
is hyperbolic relative to the collection of subgroups
$\{\mathbf{H}, \langle a \rangle\}$.
Finally, the group
$\mathbf{K}=\langle \mathbf{H} \ast \langle a \rangle, u, v \svert u^{-1} (bacb^{-1}) u =a, \ v^{-1} a v = tst^{-1} \rangle$
is hyperbolic relative to the subgroup
$\mathbf{H}$ again by Theorem~\ref{thm:osin_combination},
completing the proof of Theorem~\ref{thm:main_theorem}(i).

\section{Proof of Theorem~1.1({\rm ii})}
\label{sec:G_relatively_hyperbolic}

Let $\mathbf{H}$, $\mathbf{K}$ and $\mathbf{G}$ be the groups
defined in the statement of Theorem~\ref{thm:main_theorem}.
The aim of this section is to prove the relative hyperbolicity
of $\mathbf{G}$ with respect to the subgroup $\mathbf{H}$.
By the result of Section~\ref{sec:K_relatively_hyperbolic},
$\mathbf{K}$ is relatively hyperbolic with peripheral structure $\{\mathbf{H}\}$.

Since $\mathbf{K}$ is a multiple HNN-extension of
$\mathbf{H} \ast \langle a \rangle$ with stable letters $u$ and $v$,
the element $u$ is clearly a hyperbolic element in $\mathbf{K}$.
Moreover, we can prove the following

\medskip
\noindent {\bf Claim B.}
{\it The unique maximal elementary subgroup $E(u)$ of
$\mathbf{K}$ is precisely the cyclic subgroup $\langle u \rangle$.}

\begin{proof}
To find $E(u)$, view $\mathbf{K}$ as an HNN-extension with stable letter $u$ of
\[
\mathbf{L}:=\langle \mathbf{H} \ast \langle a \rangle, v \rangle \le \mathbf{K}.
\]
Suppose to the contrary that $E(u) \setminus \langle u \rangle \neq \emptyset$.
Among all such elements in $E(u) \setminus \langle u \rangle$,
we take an element, say $f$,  in $u$-reduced form with minimal number of $u^{\pm 1}$.
For such $f$, clearly $fu^{\pm n}f^{-1}=u^n$ for some $n \in \mathbb{N}$.
Moreover, $f$ satisfies the following.
\begin{enumerate}[\ \rm (i)]
\item $f$ does not begin with $hu$ for any $h \in \langle bacb^{-1} \rangle$,
nor with $a'u^{-1}$ for any $a' \in \langle a \rangle$;

\item $f$ does not end with $u^{-1} h$ for any $h \in \langle bacb^{-1} \rangle$,
nor with $ua'$ for any $a' \in \langle a \rangle$.
\end{enumerate}
The reason is as follows.
First, assume that $f$ begins with $hu$ for some $h \in \langle bacb^{-1} \rangle$,
that is, $f \equiv huf_1$ ($u$-reduced).
It then follows from the equality
$fu^{\pm n}f^{-1}=u^n$ that
\[
\{(u^{-1}hu) f_1\} u^{\pm n} \{f_1^{-1} (u^{-1}h^{-1}u)\}=u^n,
\]
so that
\[
(a'f_1)u^{\pm n}(f_1^{-1}a'^{-1})=u^n
\]
for some $a' \in \langle a \rangle$.
Thus $a'f_1 \in E(u) \setminus \langle u \rangle$.
But clearly $a'f_1$ has fewer number of $u^{\pm 1}$ than $f$ does,
which is a contradiction to the choice of $f$.
Next, assume that $f$ begins with $a'u^{-1}$ for some $a' \in \langle a \rangle$,
that is, $f \equiv a'u^{-1}f_2$ ($u$-reduced).
Then
\[
\{(ua'u^{-1}) f_2\} u^{\pm n} \{f_2^{-1} (u{a'}^{-1}u^{-1})\}=u^n,
\]
and so
\[
(hf_2)u^{\pm n}(f_2^{-1}h^{-1})=u^n
\]
for some $h \in \langle bacb^{-1} \rangle$.
Hence $hf_2 \in E(u) \setminus \langle u \rangle$.
In addition, $hf_2$ has fewer number of $u^{\pm 1}$ than $f$ does,
which is also a contradiction to the choice of $f$.
Therefore, (i) holds.

For (ii), note that $f^{-1} \in E(u) \setminus \langle u \rangle$
with the same number of $u^{\pm 1}$ as $f$ has.
Also, if $f$ ends with $u^{-1} h$ for some $h \in \langle bacb^{-1} \rangle$,
or with $ua'$ for any $a' \in \langle a \rangle$,
then
$f^{-1}$ begins with $h^{-1}u$ or with ${a'}^{-1}u^{-1}$.
But then by the same argument applied to $f$ in the proof of (i),
we reach a contradiction. So (ii) holds.

But then the expression $fu^{\pm n}f^{-1}u^{-n}$
cannot not represent the identity element in $\mathbf{K}$ by Britton's Lemma.
This contradiction completes the proof of the claim.
\end{proof}

By Claim~B together with Theorem~\ref{thm:hyperbolically_embedded},
the peripheral structure $\mathbf{K}$ can be extended to
$\{\mathbf{H}, \langle u \rangle \}$.
In this point of view, $v$ is a hyperbolic element in $\mathbf{K}$,
since $v$ has infinite order, and since is not conjugate to any element of $\mathbf{H}$
nor to any element of $\langle u \rangle$.
Moreover,
the unique maximal elementary subgroup $E(v)$ of
$\mathbf{K}$ is precisely the cyclic subgroup $\langle u \rangle$.
To see this,
view $\mathbf{K}$ as an HNN-extension with stable letter $v$ of
\[
\mathbf{M}:=\langle \mathbf{H} \ast \langle a \rangle, u \rangle \le \mathbf{K},
\]
and then apply a similar argument as in the proof of Claim~B.
This together with Theorem~\ref{thm:hyperbolically_embedded} again,
the peripheral structure of $\mathbf{K}$ can be extended further to
the collection of subgroups
$\{\mathbf{H}, \langle u \rangle, \langle v \rangle \}$.

Clearly there exist monomorphisms
$\iota: \langle u \rangle \rightarrow \mathbf{H}$ and $\zeta:\langle v \rangle \rightarrow \mathbf{H}$
defined by
$\iota(u)=c^3ec^3e^{-1}$ and $\zeta(v)=c^3fc^3f^{-1}$.
Then by applying Theorem~\ref{thm:osin_combination} twice,
we finally obtain that
$\mathbf{G}=\langle \mathbf{K}, x, y \svert x^{-1}ux=c^3ec^3e^{-1}, \ y^{-1}vy=c^3fc^3f^{-1} \rangle$
is hyperbolic relative to $\mathbf{H}$. This completes the proof of Theorem~\ref{thm:main_theorem}(ii).

\section{Proof of Theorem~1.1({\rm iii})}
\label{sec:non_Hopficity}

Let $\mathbf{G}$ be the group defined in the statement of Theorem~\ref{thm:main_theorem}.
The aim of this section is to prove that
$\mathbf{G}$ is a non-Hopfian group
by constructing a surjective, but not injective, endomorphism of $\mathbf{G}$.

Let $\psi$ be a unique homomorphism from a free group with basis
$\{b, c, s, t, e, f, a, u,$ $v, x, y\}$ to $\mathbf{G}$ induced by the mapping
\[
\begin{aligned}
b &\mapsto b, \quad c \mapsto c^3, \quad s \mapsto s^3, \quad t \mapsto t, \quad e \mapsto e, \quad f \mapsto f, \\
a &\mapsto s, \quad u \mapsto 1, \quad v \mapsto 1, \quad x \mapsto x \quad \text{and} \quad y \mapsto y.
\end{aligned}
\]
Then it is easy to see that
every defining relator in presentations (\ref{equ:H_0_presentation})--(\ref{equ:G_presentation})
is sent to the identity element in $\mathbf{G}$ by $\psi$.
So $\psi$ induces an endomorphism $\tilde{\psi}$ of $\mathbf{G}$.

We will show that $\tilde{\psi}$ is surjective but not injective.
To prove that $\tilde{\psi}$ is surjective,
it is sufficient to show that
$a, c, u, v \in \im \, \tilde{\psi}$.
From the $x$- and $y$-relations in presentation (\ref{equ:G_presentation}),
it follows that $u, v \in \im \, \tilde{\psi}$.
This together with the $v$-relation in presentation (\ref{equ:K_presentation})
yields that $a \in \im \, \tilde{\psi}$,
so that $c \in \im \, \tilde{\psi}$ from the $u$-relation in presentation (\ref{equ:K_presentation}).
Therefore, $\tilde{\psi}$ is a surjective endomorphism of $\mathbf{G}$.

However, $\tilde{\psi}$ is not injective,
since $c^3 \in \ker \, \tilde{\psi}$ but $c^3 \stackrel{\mathbf{G}}{\neq} 1$.
Indeed,
$c^3$ is contained in the finite subgroup $\mathbf{H}_0$ of $\mathbf{G}$;
so if $c^3 \stackrel{\mathbf{G}}{=} 1$, then $c^3 \stackrel{\mathbf{H}_0}{=} 1$,
which is obviously a contradiction.

\section{Proof of Theorem~1.1({\rm iv})}
\label{sec:Hopficity}

Let $\mathbf{H}$ be the group defined in the statement of Theorem~\ref{thm:main_theorem}.
The aim of this section is to prove that $\mathbf{H}$ is a Hopfian group.
Since the free product of two finitely generated Hopfian groups is also Hopfian (see \cite{Dey}),
it suffices to prove that $\mathbf{H}_2$ is Hopfian.

Throughout this section, let $\varphi$ be a surjective endomorphism of $\mathbf{H}_2$.
We begin with the following

\begin{proposition}
	\label{prop:phi_b_c}
	We may assume that $\varphi(b)=b$ and $\varphi(c)=c^k$ with $k \in \mathbb{Z}$.
\end{proposition}

\begin{proof}
Clearly $\varphi(\mathbf{H}_0)$ is a finite subgroup of $\mathbf{H}_2$.
Since $\mathbf{H}_0$ is a maximal finite subgroup of $\mathbf{H}_2$,
$\varphi(\mathbf{H}_0)$ is contained in a conjugate of $\mathbf{H}_0$.
So by replacing $\varphi$ with the composition of $\varphi$ and an appropriate inner automorphism of $\mathbf{H}_2$, we may assume that
$\varphi(\mathbf{H}_0) \subseteq \mathbf{H}_0$,
so that $\varphi(b), \ \varphi(c) \in \mathbf{H}_0$.

First consider $\varphi(c)$.
Note that every element in $\mathbf{H}_0$ can be written as $c^k$ or $bc^k$, where $k \in \mathbb{Z}$.
Since $c^9=1$ in $\mathbf{H}_2$, $\varphi(c)^9=\varphi(c^9)=1$.
But since $(bc^k)^2=1$ for every $k \in \mathbb{Z}$,
we must have $\varphi(c)=c^{k'}$ for some $k' \in \mathbb{Z}$.

Next consider $\varphi(b)$.
Since $b^2=1$ in $\mathbf{H}_2$, $\varphi(b)^2=\varphi(b^2)=1$.
As observed above,
$\varphi(b)=1$ or otherwise $\varphi(b)=bc^m$ in $\mathbf{H}_2$ for some $m \in \mathbb{Z}$.
Assume that $\varphi(b)=1$.
Then since $\varphi$ is onto and $\varphi(c)=c^{k'}$,
	we have
	\[
	\langle \varphi(s), \varphi(t) \rangle/\llangle c \rrangle_{\mathbf{H}_2}=\langle \varphi(b), \varphi(c), \varphi(s), \varphi(t) \rangle/\llangle c \rrangle_{\mathbf{H}_2} = \mathbf{H}_2/\llangle c \rrangle_{\mathbf{H}_2}.
	\]
	This implies that the quotient group $\mathbf{H}_2/\llangle c \rrangle_{\mathbf{H}_2}$ can be generated by two elements.
	But then,
	\[
	\begin{aligned}
		\mathbf{H}_2/\llangle c, \, s^2, \, t^2, \, t^{-1}btb^{-1}\rrangle_{\mathbf{H}_2}
		&\cong \langle \bar{b} \rangle \times \langle \bar{s} \rangle \times \langle \bar{t} \rangle \\
		&\cong \mathbb{Z}_2 \times \mathbb{Z}_2 \times \mathbb{Z}_2
	\end{aligned}
	\]
	could be also generated by two elements, which is a contradiction.
	Hence $\varphi(b) \neq 1$, and thus $\varphi(b)=bc^m$.
	
	Now let $\rho$ be the unique homomorphism from a free group with basis $\{b, c, s, t\}$ to $\mathbf{H}_2$ induced by the mapping
	\[
	\rho: b \mapsto bc^{-m}, \ \ c \mapsto c, \ \ s \mapsto s \ \ \textrm{and} \ \ t \mapsto t.
	\]
	Then every defining relator in presentations~(\ref{equ:H_0_presentation})--(\ref{equ:H_2_presentation}) is sent to the identity element in $\mathbf{H}_2$ by $\rho$.
	So $\rho$ induces an endomorphism $\tilde{\rho}$ of $\mathbf{H}_2$.
	Also, since $m$ is an arbitrary integer, there exists an endomorphism $\tilde{\rho}'$ of $\mathbf{H}_2$ defined by $b \mapsto bc^m$, $c \mapsto c$, $s \mapsto s$ and $t \mapsto t$.
	Then clearly $\tilde{\rho} \circ \tilde{\rho}'=id_{\mathbf{H}_2}$ and $\tilde{\rho}' \circ \tilde{\rho}=id_{\mathbf{H}_2}$. This means that $\tilde{\rho}$ is an automorphism of $\mathbf{H}_2$.
	By replacing further $\varphi$ with the composition $\tilde{\rho} \circ \varphi$,
	we may finally assume that
	$\varphi(b)=b$ and $\varphi(c)=c^{k'}$ with $k' \in \mathbb{Z}$, as desired.
\end{proof}

\begin{proposition}
	\label{prop:phi_s_1}
	Under the assumption that $\varphi(b)=b$ and $\varphi(c)=c^k$ with $k \in \mathbb{Z}$,
	we may further assume that $\varphi(s)=s^{\pm 3^p}$ with $p \in \mathbb{Z}_+ \cup \{0\}$
	and that
	\[
	\varphi(t) \equiv t^{\epsilon_1} w_1 \cdots t^{\epsilon_h} w_h \quad (\textrm{$t$-reduced}),
	\]
	where $h \ge 1$, $\epsilon_i=\pm 1$ and $w_i \in \mathbf{H}_1$ for every $i=1, \dots, h$.
\end{proposition}

\begin{proof}
Let $\Gamma$ be the Bass-Serre tree associated to $\mathbf{H}_2$
viewed as an HNN-extension of $\mathbf{H}_1$ (see Theorem~\ref{thm:Bass_Serre_HNN_extension}).
Denote by $v$ the vertex labeled as the coset $\mathbf{H}_1$.
Since $\varphi(\mathbf{H}_0) \subseteq \mathbf{H}_0 \subseteq \mathbf{H}_1$,
clearly $\varphi(\mathbf{H}_0)$ fixes $v$.
We shall show that $\varphi(s)$ fixes $v$ as well.
Assume on the contrary that $\varphi(s)v \neq v$.
Clearly,
$\varphi(s)\varphi(\mathbf{H}_0)\varphi(s)^{-1}$ fixes $\varphi(s)v$.
It also fixes $v$, since
\[
\varphi(s)\varphi(\mathbf{H}_0)\varphi(s)^{-1}=\varphi(s \mathbf{H}_0 s^{-1})=\varphi(\mathbf{H}_0) \subseteq \mathbf{H}_1.
\]
Since $\varphi(s)v \neq v$, and since $T$ is a tree, $\varphi(s)\varphi(\mathbf{H}_0)\varphi(s)^{-1}$
fixes an edge in $\Gamma$, and thus it is a subgroup of some edge stabilizer of $\Gamma$.
But since every edge stabilizer of $\Gamma$ is conjugated to $\langle s \rangle$ which is an infinite cyclic group, and since $\varphi(s)\varphi(\mathbf{H}_0)\varphi(s)^{-1}$ is finite,
the only possibility is that $\varphi(\mathbf{H}_0)=\{1\}$,
which is impossible.
Therefore, $\varphi(s)$ fixes $v$, that is, $\varphi(s) \in \mathbf{H}_1$.

Now consider $\varphi(t)$.
Since $\varphi$ is onto, $\varphi(t) \notin \mathbf{H}_1$.
This yields that $\varphi(t)v \neq v$.
Since $st=ts^3$ in $\mathbf{H}_2$,
$\varphi(s)\varphi(t)=\varphi(t)\varphi(s)^3$, and hence
\[
\varphi(s)\varphi(t)v=\varphi(t)\varphi(s)^3v=\varphi(t)v.
\]
This means that $\varphi(s)$ fixes $\varphi(t)v$.
Since $\varphi(t)v \neq v$, $\varphi(s)$ fixes an edge in $\Gamma$.
In particular, $\varphi(s)$ fixes every edge on the geodesic $[v, \varphi(t)v]$ in $\Gamma$.

At this point, write $\varphi(t)$ as
	\[
	\varphi(t) \equiv w_0 t^{\epsilon_1} w_1 \cdots t^{\epsilon_h} w_h \quad \textrm{($t$-reduced)},
	\]
	where $h \ge 1$, $\epsilon_i=\pm 1$ for every $i=1, \dots, h$,
	and $w_j \in \mathbf{H}_1$ for every $j=0, \dots, h$.
	Thus by replacing $\varphi$ with the composition of $\varphi$ and the inner automorphism of $\mathbf{H}_2$ given by ${w_0}^{-1}$,
	we may assume that
	$\varphi(b)=bc^m$, $\varphi(c)=c^{k'}$, $\varphi(s) \in \mathbf{H}_1$ and
	\[
	\varphi(t) \equiv t^{\epsilon_1} w_1 \cdots t^{\epsilon_h}{w_h}' \quad \textrm{($t$-reduced)},
	\]
	where $h \ge 1$, $\epsilon_i=\pm 1$ for every $i=1, \dots, h$,
	and $w_h'$, $w_j \in \mathbf{H}_1$ for every $j=1, \dots, h-1$.
	Then by replacing further $\varphi$ with the composition $\tilde{\rho} \circ \varphi$,
	where $\tilde{\rho}$ is the automorphism of $\mathbf{H}_2$ (see the proof of Proposition~\ref{prop:phi_b_c}) defined by
	\[
	\tilde{\rho}: b \mapsto bc^{-m}, \ \ c \mapsto c, \ \ s \mapsto s \ \ \textrm{and} \ \ t \mapsto t,
	\]
	we may further assume that
	$\varphi(b)=b$ and $\varphi(c)=c^{k'}$, $\varphi(s) \in \mathbf{H}_1$ and $\varphi(t)$ as above.

Now recall that $\varphi(s)$ fixes every edge on the geodesic $[v, \varphi(t)v]$ in $\Gamma$.
In particular, $\varphi(s)$ belongs to the stabilizer of the first edge
on the geodesic $[v, \varphi(t)v]$.
Since $\varphi(t)$ begins with the letter $t$ or $t^{-1}$,
the first edge on the geodesic $[v, \varphi(t)v]$ is labeled as $\langle s \rangle$
or as $t^{-1} \langle s \rangle$.
So the stabilizer of the former is $\langle s \rangle$
and that of the latter is $t^{-1}\langle s \rangle t=\langle s^3 \rangle$.
In either case, $\varphi(s) \in \langle s \rangle$, that is, $\varphi(s)=s^r$ for some $r \in \mathbb{Z}$.

Write $r=\pm 3^p q$, where $p, q \in \mathbb{Z}_+ \cup \{0\}$
	and $\gcd (3,q)=1$.
	We shall show that $q=1$.
	Clearly, $s^r=s^{\pm 3^p q} \in \llangle s^q \rrangle_{\mathbf{H}_2}$.
	Suppose that $q \neq 1$, namely $q \ge 2$.
	Since $\varphi$ is onto and since $\varphi(b)=b$ and $\varphi(c)=c^k$, we have
	\[
	\langle \varphi(t) \rangle / \llangle b, c, s^q \rrangle_{\mathbf{H}_2}
	=\langle \varphi(b), \varphi(c), \varphi(s), \varphi(t) \rangle / \llangle b, c, s^q \rrangle_{\mathbf{H}_2}
	=\mathbf{H}_2 / \llangle b, c, s^q \rrangle_{\mathbf{H}_2}.
	\]
	This implies that the quotient group $\mathbf{H}_2 / \llangle b, c, s^q \rrangle_{\mathbf{H}_2}$ can be generated by only one element.
	But this is a contradiction, since
	\[
	\mathbf{H}_2 / \llangle b, c, s^q \rrangle_{\mathbf{H}_2} \cong \langle \bar{s} \rangle \times \langle \bar{t} \rangle
	\cong \mathbb{Z}_q \rtimes \mathbb{Z}.
	\]
	Hence $q=1$, that is, $\varphi(s)=s^{\pm 3^p}$, as desired.
\end{proof}

\begin{proposition}
	\label{prop:phi_c}
	Under the assumption that $\varphi(b)=b$, $\varphi(c)=c^k$,
	$\varphi(s)=s^{\pm 3^p}$ and
	$\varphi(t)=t^{\epsilon_1} w_1 \cdots t^{\epsilon_h} w_h$ $(t$-reduced$)$,
	where $k, p \in \mathbb{Z}$ with $p \ge 0$, $h \ge 1$,
	$\epsilon_i=\pm 1$, and where $w_i \in \mathbf{H}_1$
	for every $i=1, \dots, h$,
	we may further assume that $\varphi(c)=c$ and $\varphi(s)=s$.
\end{proposition}

\begin{proof}
	We shall first show that $\varphi(c)=c^k$ with $k \equiv \pm 1 \pmod 3$.
	Suppose the contrary.
	Then $\varphi(c)=1$ or $\varphi(c)=c^{\pm 3}$, because $c^9=1$.
	Since $\varphi$ is onto,
	we have
	\[
	\langle \varphi(t) \rangle/\llangle b, c^3, s \rrangle_{\mathbf{H}_2}
	=\langle \varphi(b), \varphi(c), \varphi(s), \varphi(t) \rangle/\llangle b, c^3, s \rrangle_{\mathbf{H}_2}
	= \mathbf{H}_2/\llangle b, c^3, s \rrangle_{\mathbf{H}_2}.
	\]
	This implies that the quotient group $\mathbf{H}_2/\llangle b, c^3, s \rrangle_{\mathbf{H}_2}$ can be generated by only one element.
	But since
	\begin{equation}
		\label{equ:three_generators}
		\begin{aligned}
			\mathbf{H}_2/\llangle b, c^3, \, s \rrangle_{\mathbf{H}_2}
			\cong \langle \bar{c} \rangle \ast \langle \bar{t} \rangle
			\cong \mathbb{Z}_3  \ast \mathbb{Z},
		\end{aligned}
	\end{equation}
	we reach a contradiction.
	
	Thus $\varphi(c)=c^k$ with $k \equiv \pm 1 \pmod 3$.
	Let $k'$ be an integer such that $kk' \equiv 1 \ \pmod 9$.
	For such $k'$, clearly either $c^{-3k'}=c^{-3}$ or $c^{-3k'}=c^3$.
	Moreover, $c^{-3k'}=c^{-3k}$.
	First assume that $c^{-3k'}=c^{-3k}=c^{-3}$.
	Let $\rho$ and $\rho'$ be the unique homomorphisms from a free group with basis $\{b, c, s, t\}$ to $\mathbf{H}_2$ induced by the mappings
	\[
	\begin{aligned}
		\rho&: b \mapsto b, \ \ c \mapsto c^{k'}, \ \ s \mapsto s, \ \ t \mapsto t, \ \ \text{and} \\
		\rho'&: b \mapsto b, \ \ c \mapsto c^k, \ \ s \mapsto s, \ \  t \mapsto t.
	\end{aligned}
	\]
	Then every defining relator in presentations~(\ref{equ:H_0_presentation})--(\ref{equ:H_2_presentation}) is sent to the identity element in $\mathbf{H}_2$ by both $\rho$ and $\rho'$.
	So there are endomorphisms $\tilde{\rho}$ and $\tilde{\rho}'$  of $\mathbf{H}_2$ induced by $\rho$ and $\rho'$, respectively.
	For such $\tilde{\rho}$ and $\tilde{\rho}'$, clearly $\tilde{\rho} \circ \tilde{\rho}'=id_{\mathbf{H}_2}$ and $\tilde{\rho}' \circ \tilde{\rho}=id_{\mathbf{H}_2}$, meaning that $\tilde{\rho}$ is an automorphism of $\mathbf{H}_2$.
	By replacing $\varphi$ with $\varphi \circ \tilde{\rho}$,
	we may assume that
	$\varphi(b)=b$, $\varphi(c)=c$ and $\varphi(s)=s^{\pm 3^p}$.
	
	Next assume that $c^{-3k'}=c^{-3k}=c^3$.
	Let $\tau$ and $\tau'$ be the unique homomorphisms from a free group with basis $\{b, c, s, t\}$ to $\mathbf{H}_2$ induced by the mappings
	\[
	\begin{aligned}
		\tau&: b \mapsto b, \ \ c \mapsto c^{k'}, \ \ s \mapsto s^{-1}, \ \ t \mapsto t, \ \ \text{and} \\
		\tau'&: b \mapsto b, \ \ c \mapsto c^k, \ \ s \mapsto s^{-1}, \ \  t \mapsto t.
	\end{aligned}
	\]
	Again since every defining relator in presentations~(\ref{equ:H_0_presentation})--(\ref{equ:H_2_presentation}) is sent to
	the identity element in $\mathbf{H}_2$ by both $\tau$ and $\tau'$, there are endomorphisms $\tilde{\tau}$ and $\tilde{\tau'}$ of $\mathbf{H}_2$ induced by $\tau$ and $\tau'$, respectively.
	For such $\tilde{\tau}$ and $\tilde{\tau}'$, clearly $\tilde{\tau} \circ \tilde{\tau}'=id_{\mathbf{H}_2}$ and $\tilde{\tau}' \circ \tilde{\tau}=id_{\mathbf{H}_2}$, meaning that $\tilde{\tau}$ is an automorphism of $\mathbf{H}_2$.
	By replacing $\varphi$ with $\varphi \circ \tilde{\tau}$, we may also assume that
	$\varphi(b)=b$, $\varphi(c)=c$ and $\varphi(s)=s^{\pm 3^{p}}$.

Now consider $\varphi(s)$.
	From the defining relation $s^{-1}bs=bc^{-3}$ in presentation (\ref{equ:H_1_presentation}),
	it follows that
	$\varphi(s)^{-1}\varphi(b)\varphi(s)=\varphi(b)\varphi(c)^{-3}$.
	Here, since $\varphi(b)=b$, $\varphi(c)=c$ and $\varphi(s)=s^{\pm 3^p}$,
	we get
	\begin{equation}
		\label{equ:equ_s}
		s^{\mp 3^p} b s^{\pm 3^p}=bc^{-3}.
	\end{equation}
	But from the defining relations in presentation (\ref{equ:H_1_presentation}),
we obtain that $s b s^{-1}=bc^3$ and that
$s^{\mp 3^i} b s^{\pm 3^i}=b$ for any $i \ge 1$.
Combining these with (\ref{equ:equ_s})
yields $c^6=1$ or $c^3=1$, contrary to the fact that $c$ has order $9$.
Therefore,
the only possibility to avoid a contradiction is that $\varphi(s)=s$, as desired.
\end{proof}

\begin{proposition}
	\label{prop:phi_t_2}
	Under the assumption that $\varphi(b)=b$, $\varphi(c)=c$,
	$\varphi(s)=s$
	and $\varphi(t)=t^{\epsilon_1} w_1 \cdots t^{\epsilon_h} w_h$ $(t$-reduced$)$,
	where $\epsilon_i=\pm 1$ and $w_i \in \mathbf{H}_1$
	for every $i=1, \dots, h$,
	we may further assume that
	$\varphi(t)=t^{\epsilon_1} w_1 \cdots w_{h-1} t^{\epsilon_h}$.
\end{proposition}

\begin{proof}
	Define the unique homomorphism $\rho$ from a free group with basis $\{b, c, s, t\}$ to $\mathbf{H}_2$ induced by the mapping
	\[
	\rho: b \mapsto b, \ \ c \mapsto c, \ \ s \mapsto s \ \ \textrm{and} \ \ t \mapsto t w_h^{-1}.
	\]
	Then every defining relator in prsentation~(\ref{equ:H_0_presentation})--(\ref{equ:H_2_presentation}) is sent to the identity element in $\mathbf{H}_2$ by $\rho$, since $\rho(t)^{-1}s\rho(t)s^{-3} \equiv w_ht^{-1}stw_h^{-1}s^{-3} \stackrel{\mathbf{H}_2}{=} w_hs^3w_h^{-1}s^{-3} \stackrel{\mathbf{H}_2}{=} 1$. So $\rho$ induces an endomorphism $\tilde{\rho}$ of $\mathbf{H}_2$.
	For the same reason, there exists the endomorphism $\tilde{\rho}'$ of $\mathbf{H}_2$ defined by
	$b \mapsto b$, $c \mapsto c$, $s \mapsto s$ and $t \mapsto tw_h$.
	Then clearly $\tilde{\rho} \circ \tilde{\rho}'=id_{\mathbf{H}_2}$ and $\tilde{\rho}' \circ \tilde{\rho}=id_{\mathbf{H}_2}$. Therefore, $\tilde{\rho}$ is an automorphism of $\mathbf{H}_2$.
	By replacing $\varphi$ with the composition $\varphi \circ \tilde{\rho}$,
	we may assume that $\varphi(b)=b$, $\varphi(c)=c$ and
	$\varphi(s)=s$, and that $\varphi(t)=t^{\epsilon_1} w_1 \cdots w_{h-1} t^{\epsilon_h}$, as desired.
\end{proof}

\begin{proposition}
	\label{prop:phi_t_3}
	Under the assumption that $\varphi(b)=b$, $\varphi(c)=c$,
	$\varphi(s)=s$
	and $\varphi(t)=t^{\epsilon_1} w_1 \cdots w_{h-1} t^{\epsilon_h}$ $(t$-reduced$)$,
	where $\epsilon_i=\pm 1$ for every $i=1, \dots, h$,
	and $w_j \in \mathbf{H}_1$ for every $j=1, \dots, h-1$,
	we may finally assume that
	$\varphi(t)=t$.
\end{proposition}

\begin{proof}
	From the defining relation $t^{-1}st=s^3$ in presentation (\ref{equ:H_2_presentation})
	together with the hypothesis,
	it follows that
	\begin{equation}
		\label{equ:equ_t_3}
		(t^{-\epsilon_h} {w_{h-1}}^{-1} \cdots {w_1}^{-1} t^{-\epsilon_1}) s (t^{\epsilon_1} w_1 \cdots w_{h-1} t^{\epsilon_h}) s^{-3}=1.
	\end{equation}
	By Britton's Lemma, this expression is not $t$-cyclically reduced.
	
	Assume that $\epsilon_1=-1$.
	Then the part $t^{-\epsilon_1} s t^{\epsilon_1}$ is already $t$-reduced.
	So the only $t$-reductions can be made successively starting from the part $t^{\epsilon_h} s^{-3} t^{-\epsilon_h}$.
	Assume that $t$-reductions can be made only until
	$t^{\epsilon_j} (w_j \cdots (t^{\epsilon_h} s^{-3} t^{-\epsilon_h}) \cdots w_j^{-1}) t^{-\epsilon_j}$.
	Here, if $j \ge 2$, then by Britton's Lemma, equality (\ref{equ:equ_t_3}) cannot hold, a contradiction.
	Also, if $j=1$, then after making all $t$-reductions,
	$t^{\epsilon_1} (w_1 \cdots (t^{\epsilon_h} s^{-3} t^{-\epsilon_h}) \cdots w_1^{-1}) t^{-\epsilon_1}$
	becomes $s^k$  for some $k \in \mathbb{Z}$.
	In particular, since $\epsilon_1=-1$, $k$ is a multiple of $3$.
	But then $s^{k+1}=1$ from equality (\ref{equ:equ_t_3}),
	contrary to the fact that $s$ is an element of infinite order.
	
	Thus $\epsilon_1=1$.
	On the other hand, since $\varphi$ is onto,
	there is a reduced word $z(b, c, s, t)$ in $\{b, c, s, t\}$
	such that
	\begin{equation}
		\label{equ:equ_t}
		t \stackrel{\mathbf{H}_2}{=}z(b, c, s, \varphi(t)).
	\end{equation}
	We may write $z(b, c, s, t)$ as follows:
	\[
	z(b, c, s, t) \equiv z_0 t^{\delta_1} z_1 \cdots  z_{\ell-1} t^{\delta_{\ell}} z_{\ell} \quad (\textrm{$t$-reduced}),
	\]
	where $z_0, \dots, z_{\ell}$ are reduced words in $\{b, c, s\}$, $\delta_1, \dots, \delta_{\ell}=\pm 1$, and where whenever $z_i$ is not the identity element in $\mathbf{H}_1$,
	$z_i \notin \langle s \rangle$ provided either $\delta_i=-1$ or $\delta_{i+1}=1$.
	
	\medskip
	\noindent
	{\bf Claim.} {\it Even after making all $t$-reductions in the right-handed expression
		\[
		z(b, c, s, \varphi(t))=z_0 \varphi(t)^{\delta_1} z_1 \cdots  z_{\ell-1} \varphi(t)^{\delta_{\ell}} z_{\ell},
		\]
		at least one $t^{\pm 1}$ in $\varphi(t)^{\delta_i}$ remains unreduced for every $i=1, \dots, \ell$.}
	
	\begin{proof}[Proof of Claim]
		Let us consider all possible $t$-reductions in each
		$\varphi(t)^{\delta_i} z_i \varphi(t)^{\delta_{i+1}}$.
		First, if $\delta_i=-1$ and $\delta_{i+1}=1$,
		it follows from the fact $\epsilon_1=1$ that
		there is no $t$-reduction in $\varphi(t)^{-1} z_i \varphi(t)$,
		since $z_i \notin \langle s \rangle$ in this case.
		
		Next, if either both $\delta_i=-1$ and $\delta_{i+1}=-1$,
		or both $\delta_i=1$ and $\delta_{i+1}=1$,
		it follows from our assumption $z_i \notin \langle s \rangle$ in either case
		that there is no $t$-reduction in $\varphi(t)^{\mp 1} z_i \varphi(t)^{\mp 1}$.
		
		Finally, assume that $\delta_i=1$ and $\delta_{i+1}=-1$.
		In this case,
		there can be $t$-reductions in $\varphi(t) z_i \varphi(t)^{-1}$.
		Even if there are $t$-reductions in $\varphi(t) z_i \varphi(t)^{-1}$,
		not all of $t^{\pm 1}$ can be reduced.
		The reason is as follows.
		Suppose that all of $t^{\pm 1}$ in $\varphi(t) z_i \varphi(t)^{-1}$
		can be reduced.
		Then since the initial letter of $\varphi(t)$ is $t^{\epsilon_1}$ with $\epsilon_1=1$,
		the last $t$-reduction
		in $\varphi(t) z_i \varphi(t)^{-1}$ has the form $t z_i' t^{-1}$,
		where $z_i' \in \langle s^3 \rangle$,
		and so $\varphi(t) z_i \varphi(t)^{-1}$, after making all $t$-reductions,
		becomes a power of $s$.
		But then from the equality $\varphi(t)^{-1} s \varphi(t)=s^3$ in $\mathbf{H}_2$,
		it follows that $z_i \in \langle s^3 \rangle$.
		This is a contradiction to the assumption that
		$z(b,c,s,t)$ is $t$-reduced.
		
		Therefore, only in the case where $\delta_i=1$ and $\delta_{i+1}=-1$,
		$t$-reductions can happen in $\varphi(t)^{\delta_i} z_i \varphi(t)^{\delta_{i+1}}$.
		But even in this case, not all of $t^{\pm 1}$ in $\varphi(t)^{\delta_i} z_i \varphi(t)^{\delta_{i+1}}$
		can be reduced.
		Therefore, the assertion of Claim follows.
	\end{proof}
	
	In view of Claim, in order for equality (\ref{equ:equ_t}) to hold,
	we see that the only possibility is that
	$\ell=1$, so that $z(b, c, s, \varphi(t))=z_0\varphi(t)^{\delta_1}z_1$.
	Combining this with~(\ref{equ:equ_t}), we get
	\begin{equation}
		\label{equ:tt}
		t \stackrel{\mathbf{H}_2}{=} z_0\varphi(t)^{\delta_1}z_1.
	\end{equation}
	Since $\varphi(t)$ is written as a $t$-reduced form, the right-handed expression of~(\ref{equ:tt})
	is already $t$-reduced.
	Then by Britton's Lemma, for the equality in (\ref{equ:tt}) to hold,
	only one alphabet $t$ occurs
	and at the same time no alphabet $t^{-1}$ occurs
	in the right-handed expression.
	Here, since $\varphi(t)=t^{\epsilon_1} w_1 \cdots w_{h-1} t^{\epsilon_h}$ with $\epsilon_1=1$, we see that this happens only when $\delta_1=1$ and $\varphi(t)=t$, completing the proof of Proposition~\ref{prop:phi_t_3}.
\end{proof}

In conclusion, we obtain the following

\begin{corollary}
	\label{cor:hopfian}
	The group $\mathbf{H}_2$ is Hopfian.
\end{corollary}

\begin{proof}
	Let $\varphi$ be a surjective endomorphism of $\mathbf{H}_2$.
	Propositions~\ref{prop:phi_b_c}--\ref{prop:phi_t_3} show that
	the composition of $\varphi$ with appropriate automorphisms of $\mathbf{H}_2$
	becomes the identity function of $\mathbf{H}_2$,
	so that $\varphi$ is indeed an automorphism of $\mathbf{H}_2$.
	This means that $\mathbf{H}_2$ is Hopfian.
\end{proof}


\begin{thebibliography}{3}


\bibitem{Bass}
H.\ Bass, {\em Covering theory for graphs of groups},
J. Pure Appl. Algebra {\bf 89}(1--2) (1993), 3--47.







\bibitem{Coulon-Guirardel}
R.\ Coulon and V.\ Guirardel, {\em Automorphisms and endomorphisms of lacunary hyperbolic
groups}, Groups Geom. Dyn. {\bf 13} (2019), 131--148.

\bibitem{Dahmani}
F. Dahmani, {\em Combination of convergence groups}, Geom. Topol. {\bf 7} (2003), 933--963.


\bibitem{Dey}
I.\ M.\ S.\ Dey and H.\ Neumann, {\em The Hopf property of free products}, Math. Z. {\bf 117} (1970), 325--339.



\bibitem{Groves}
D.\ Groves, {\em Limit groups for relatively hyperbolic groups, II: Makanin-Razborov diagrams}, Geom. Topol. {\bf 10} (2010), 1807--1856.

\bibitem{Groves_Hull}
D.\ Groves and M.\ Hull, {\em Homomorphisms to acylindrically hyperbolic groups I: Equationally noetherian groups and families}. Trans. Amer. Math. Soc. {\bf 372}(10) (2019), 7141--7190.










\bibitem{Malcev}
A.\ I.\ Mal'cev, {\em On the faithful representation of infinite groups by matrices}, Mat. Sb. {\bf 8} (50)
(1940), 405--422.


\bibitem{Osin}
D.\ V.\ Osin, {\em Relatively hyperbolic groups: Intrinsic geometry, algebraic properties, and algorithmic problems}, Memoirs Amer. Math. Soc. {\bf 179} (2006), no. 843, vi+100 pp.

\bibitem{Osin2}
D.\ V.\ Osin, {\em Relative Dehn functions of amalgamated products and HNN-extensions}, Contemp.
Math. {\bf 394} (2006), 209--220.


\bibitem{Osin5}
D.\ V.\ Osin, {\em Elementary subgroups of relatively hyperbolic groups and bounded generation},
Internat. J. Algebra Comput. {\bf 16} (2006), no. 1, 99--118.

\bibitem{Osin4}
D. V. Osin, {\em Peripheral fillings of relatively hyperbolic groups}, Invent. Math. {\bf 167} (2007), no. 2, 295--326.


\bibitem{Rein-Weid}
C.\ Reinfeldt and R.\ Weidmann,
{\em Makanin-Razborov diagrams for hyperbolic groups}, Annales Mathematiques Blaise Pascal {\bf 26} (2019), 119--208.


\bibitem{Sela}
Z.\ Sela, {\em Endomorphisms of hyperbolic groups. I. The Hopf property}, Topology {\bf 38} (1999), 301--321.

\bibitem{Serre}
J.\ P.\ Serre, {\em Trees}, Springer Monographs in Mathematics, Springer-Verlag, Berlin, 2003, Translated from the French original by John Stillwell, Corrected 2nd printing of the 1980 English translation.


\bibitem{Strebel}
R.\ Strebel,
{\em Appendix. Small cancellation groups},
in ``Sur les groupes hyperboliques d'apr\`es Mikhael Gromov'',
Papers from the Swiss Seminar on Hyperbolic Groups held in Bern, 1988,
E.\ Ghys and P.\ de la Harpe (editors), Progr. Math., vol. {\bf 83}, Birkh\"{a}user Boston Inc., Boston, MA, 1990.


\bibitem{Wise2}
D. T. Wise, {\em A non-Hopfian automatic group}, J. Algebra {\bf 180} (1996), 845--847.

\bibitem{Yang}
W.\ Yang, {\em Notes on geometric group theory}, preprint, 2020, available at http://faculty.bicmr.pku.edu.cn/~wyang/ggt/GGTnotes.pdf
\end{thebibliography}
\end{document}